\documentclass{amsart}
\usepackage{amsmath}
\usepackage{amssymb}
\usepackage{mathrsfs}
\usepackage{stmaryrd}
\usepackage[all]{xy}
\usepackage{amsfonts}

\makeatletter \renewcommand{\everyentry@}{\vphantom{A_{[]}ˆ{[]}}}
\makeatother

\newtheorem{theorem}{Theorem}[subsection]

\newtheorem{lemma}[theorem]{Lemma}
\newtheorem{prop}[theorem]{Proposition}

\theoremstyle{definition}
\newtheorem{definition}[theorem]{Definition}
\newtheorem{example}[theorem]{Example}

\theoremstyle{remark}
\newtheorem{remark}[theorem]{Remark}

\numberwithin{equation}{subsection}

 \def \E{\mathcal E}

\def \Z {\mathbb Z}
\def \inj {\hookrightarrow }
\def \to {\rightarrow}
\def \spec \text{spec}
 \def \M{\mathfrak M}

\def \e { \underline \epsilon}

\def \p {\underline \pi}

\def  \cris {\text{cris}}
\DeclareMathOperator{\gal}{Gal}

\def \Q {\mathbb Q}
\def \t {\textnormal}
\def \Z {\mathbb Z}

\def \O {\mathcal O}
\def \gs {\mathfrak S}
\def \ur {\t{ur}}
\def \D {\mathcal D}

\def \N {\mathfrak N}

\def \v {\vee}
\def \f {\mathfrak f}
\def \Fr {\t{Fr}}

\def \gt {\mathfrak t}

\def \acris {{A_{\t{cris}}}}

\def \R {\mathcal R}

\def \sfr {\t{Mod}_{/\gs}^{r, \t{fr}}}
\def \st {\textnormal{st}}

\def \< {\left <}
\def \> {\right >}

\def \hR {\widehat \R }
\def \hM {{\hat \M}}
\begin{document}

\title{A note on lattices in semi-stable representations}

\author{Tong Liu}
\address{Department of Mathematics, University of Pennsylvania, Philadelphia,19104, USA.}
\email{tongliu@math.upenn.edu}






\begin{abstract}
Let $p\geq 3$ be a prime,  $K$ a finite extension over $\Q_p$ and $G:= \gal(\bar K
/K)$. We extend Kisin's theory on $\varphi$-modules of finite $E(u)$-height to
give a new classification of $G$-stable $\Z_p$-lattices in semi-stable representations.
\end{abstract}

\maketitle


\tableofcontents
\section{Introduction}
This note serves as a new idea to classify lattices in the
semi-stable representations.  Let $k$ be a perfect field of
characteristic $p>2$, $W(k)$ its ring of Witt vectors,
$K_0=W(k)[\frac{1}{p}]$, $K/K_0$ a finite totally ramified extension
and $G := \gal(\bar K /K)$. For many technical reasons,  we are
interested in classifying  $G$-stable $\Z_p$-lattices in
semi-stable $p$-adic Galois representations,  via linear algebra
data like admissible
 filtered $(\varphi, N)$-modules in Fontaine's theory.
Many important steps have been made  in this direction. For example,
Fontaine and Laffaille' theory \cite{fo1} on strongly divisible
$W(k)$-lattices in filtered $(\varphi, N)$-modules,  Breuil's theory
on strongly divisible $S$-lattices  (\cite{b6}, \cite{liu3}), and
Berger and Breuil's theory on Wach modules (\cite{Ber3}).
Unfortunately, these classifications always have some restrictions
(on the absolute ramification index, Hodge-Tate weights, etc). Based
on Kisin's theory in \cite{kisin2}, the aim of this paper is to
provide a classification without these restrictions (at least for
$p>2$).

More precisely, let $E(u)$ be an Eisenstein polynomial for a fixed
uniformizer $\pi$ of $K$,  $K_\infty = \cup_{n\geq1} K(
\sqrt[p^n]{\pi})$, $G_\infty = \gal(\bar K/ K_\infty)$ and $\gs=
W(k)[\![u]\!]$.  We equip $\gs$ with the endomorphism $\varphi$
which acts via Frobenius on $W(k)$, and sends $u$ to $u^p$. Let
$\sfr$ denote the category of   finite free $\gs$-modules $\M$
equipped with a $\varphi$-semi-linear map $\varphi_\M: \M \to \M $
such that the
 cokernel of $\gs$-linear map $1\otimes \varphi_\M: \gs  \otimes_{\varphi, \gs} \M \to \M$
is killed by $E(u)^r$. Objects in $\sfr$  are called
\emph{$\varphi$-modules of $E(u)$-height $r$} or \emph{Kisin
modules}. In \cite{kisin2}, Kisin proved that any $G_\infty$-stable
$\Z_p$-lattice $T$ in a semi-stable Galois representation comes from
a Kisin module (See Theorem \ref{kisin} for details). Obviously,
extra data have to be added if one would like to extend the
classification of $G_\infty$-stable lattices to the classification
of $G$-stable lattices. Our ideal is to imitate the theory of
$(\varphi, \Gamma)$-modules. But $G_\infty$ is not a normal subgroup
of $G$ and there is no natural $G$-action on $\gs$. To remedy this,
we construct a $\gs$-algebra $\hR$ inside $W(R)$ such that $\hR$ is
stable under Frobenius and the $G$-action. Furthermore, the
$G$-action on $\hR$ factors through $\hat G := \gal(K_{\infty, p^\infty }/K)$
where $K_{\infty, p^\infty }$ is the Galois closure of $K_\infty$.
The construction of $\hR$ allows us to define $(\varphi, \hat
G)$-module which is Kisin module $(\M, \varphi)$ with extra
semi-linear $\hat G$-action on $\hR \otimes_{\varphi, \gs} \M$
compatible with Frobenius (see Definition \ref{definition} for
details). Our main result in this note is that the category of
$G$-stable $\Z_p$-lattices in semi-stable representations with
Hodge-Tate weights in $\{0, \dots, r \}$ is anti-equivalent to the
category of $(\varphi, \hat G)$-modules of $E(u)$-height $r$.

Just as any integral version of $p$-adic Hodge theory before,
$(\varphi, \hat G)$-modules will help us better to understand the
reduction of semi-sable representations, and this will be discussed
in forthcoming work. On the other hand, so far we do not fully
understand the structure of $\hR$. In fact, $\hR$ seems quite
complicated (See Example \ref{example}). So at least at this stage,
it seems that our theory only serves as a theoretic approach.
 We hope we can simplify  this theory in
the future by further exploring the structure of $\hR$,  such that
we could  provide more explicit examples or carry out some concrete
computations by $(\varphi, \hat G)$-modules.

\section{Preliminary and the Main Result}\label{sa2}
\subsection{Kisin Modules}
Recall that $k$ is a perfect field of characteristic $p>2$, $W(k)$
its ring of Witt vectors, $K_0=W(k)[\frac{1}{p}]$, $K/K_0$ a finite
totally ramified extension and $e=e(K/K_0)$ the absolute
ramification index. Throughout this paper we fix a uniformiser
$\pi\in K$ with Eisenstein polynomial $E(u)$. Recall that $\gs =
W(k)\llbracket u\rrbracket$ is equipped with a Frobenius
endomorphism $\varphi$ via $u\mapsto u^p$ and the natural Frobenius
on $W(k)$. Throughout this paper we reserve
$\varphi$ to denote various Frobenius structures.
A \emph{$\varphi$-module} (over $\gs$) is an $\gs$-module
$\M$ equipped with a $\varphi$-semi-linear map $\varphi: \M \to \M
$. A morphism between two objects $(\M_1, \varphi_1)$, $(\M_2,
\varphi_2)$ is a $\gs$-linear morphism compatible with the
$\varphi_i$. Denote by $\sfr$ the category of
\emph{$\varphi$-modules of  $E(u)$-height $r$}  in the sense
that  $\mathfrak M $ is finite  free \footnote{This is a somewhat ad
hoc definition because we only concern finite free $\gs$-modules
here. In fact, one may only require that $\M$ is of $\gs$-finite
type when define $\varphi$-modules of finite $E(u)$-height,
especially, when study $p$-power torsion representations.}
 over $\gs$ and the cokernel of $\varphi^*$ is
killed by $E(u)^r$, where $\varphi^*$ is the $\gs$-linear map
$1\otimes \varphi: \gs \otimes_{\varphi, \gs}\M \to \M $. Object in
$\sfr$ is also called \emph{Kisin module $($of height
\footnote{Throughout this paper, the height is always $E(u)$-height.
So we always omit ``$E(u)$".} $r)$}.

Let $R= \varprojlim \O_{\bar K }/p$ where the transition maps are
given by Frobenius.  By the universal property of  the Witt vectors
$W(R)$ of $R$, there is a unique surjective projection map $\theta :
W(R) \to \widehat \O_{\bar K}$ to the $p$-adic completion of
$\O_{\bar K}$, which lifts the  the projection $R \to \O_{\bar K}/ p$
onto the first factor in the inverse limit.
Let $\pi_n\in \bar K $ be a $p^n$-th root of $\pi$, such that
$(\pi_{n+1})^p=\pi_n$; write $\underline \pi=(\pi_n)_{n\geq 0}\in R$
and let $[\underline \pi ]\in W(R)$ be the Techm\"uller
representative. We embed the $W(k)$-algebra $W(k)[u]$ into $W(R)$ by
the map $u\mapsto [\underline \pi]$. This embedding extends to an
embedding $\gs \inj W(R)$, and, as $\theta([\p])= \pi$, $\theta|_
\gs $ is the map $\gs \to \O_K $ sending $u$ to $\pi$. This
embedding is compatible with Frobenious endomorphisms.

Denote by $\O_\E$ the $p$-adic completion of $\gs[\frac{1}{u}]$.
Then $\O_\E $ is a discrete valuation ring with residue field
the Laurent series ring $k(\!(u)\!)$. We write $\E$ for the field of
fractions of $\O_\E$. If $\Fr R$ denotes the field of fractions
of $R$, then the inclusion $\gs \inj W(R)$ extends to an inclusion
$\O_\E\inj W(\Fr R)$. Let $\E^{\t{ur}}\subset W(\t{Fr}R)[\frac{1}{p}]$
denote the maximal unramified extension of $\E$ contained in
$W(\Fr R)[\frac{1}{p}]$, and $\O^{\t{ur}} $  its ring of integers.
Since $\Fr R$ is easily seen to be algebraically  closed, the residue field
$\O^{\t{ur}}/p\O^{\t{ur}}$ is the separable closure of $k(\!(u)\!)$. We
denote by $\widehat {\E^{\ur}}$ the $p$-adic completion of $\E^{\ur}$,
and by $\widehat{ \O^{\ur}}$ its ring of integers. $\widehat{\E^{\ur}}$
is also equal to the closure of $\E ^\ur $ in $W(\Fr R)[\frac{1}{p}]$.
We write $\gs^\ur =\widehat{ \O^{ \ur }}\cap W(R)\subset W(\Fr R)$.
We regard all these rings as subrings of $W(\Fr R)[\frac{1}{p}]$.

Recall that $K_\infty =\bigcup_{n\geq 0}K(\pi_n)$, and $G_\infty=
\t{Gal}(\bar K /K_\infty)$.  $G_\infty$  acts continuously on $\gs
^\ur $ and $\E^\ur$ and
 fixes the subring $\gs
\subset W(R)$. Finally, we
denote by $\t{Rep}_{\Z_p}(G_\infty)$ the category of continuous
$\Z_p$-linear representations of $G_\infty$ on finite free $\Z_p$-modules.

 For any Kisin module $(\M, \varphi)$, one can associate a $\Z_p[G_\infty]$-module:
 $$T_\gs(\M):= \t{Hom}_{\gs, \varphi}(\M, \gs^\ur).$$
One can show that $T_\gs(\M)$ is finite free over $\Z_p$ and
$\t{rank}_{\Z_p}(T_\gs(\M)) = \t{rank}_{\gs}(\M)$ (see for example,
Corollary (2.1.4) in \cite{kisin2}). Let $V$ be a continuous linear
representation of $G:= \gal(\bar K /K)$ on a finite dimensional
$\Q_p$-vector space. $V$ is called of \emph{$E(u)$-height $r$} if
there exists a $G_\infty$-stable $\Z_p$-lattices $T\subset V$ and  a
Kisin module $\M \in \sfr$ such that  $T \simeq T_\gs(\M)$. We refer
\cite{fo7}  to the notion of \emph{semi-stable} $p$-adic
representations\footnote{A $p$-adic representation $V$ is called
semi-stable if $\t{dim}_{K_0}(V \otimes_{\Q_p}B_\st)^G=
\t{dim}_{\Q_p}V$. See \cite{fo3} for the construction of $B_\st$.}.
 The following theorem summarizes the
known results on the relation between semi-stable representations and representations of finite $E(u)$-height.
\begin{theorem}[\cite{kisin2}]\label{kisin}\begin{enumerate} \item The functor $T_\gs : \sfr\to \t{Rep}_{\Z_p}(G_\infty)$ is fully faithful.
\item A semi-stable representation with Hodge-Tate weights in $\{0, \dots, r\}$ is of finite $E(u)$-height $r$.
\end{enumerate}
\end{theorem}
\begin{remark} \begin{enumerate}
\item Suppose that $V$ is of $E(u)$-height $r$. Then it is easy
to show that any $G_\infty$-stable $\Z_p$-lattice $T \subset V$
comes from a Kisin module  $\N \in \sfr$, i.e., $T \simeq
T_\gs(\N)$. See the proof of Lemma (2.1.15) in \cite{kisin2}.

\item  It is natural to ask if the converse question for Theorem
\ref{kisin} (2) is true. As we will see later, our results in this
note may be regarded as partial results in this direction.
\end{enumerate}
\end{remark}

\subsection{ ($\varphi, \hat G$)-modules }\label{S2}
We denote  by $S$  the $p$-adic completion of the divided power
envelope of $W(k)[u]$ with respect to the ideal generated by $E(u)$.
There is a unique map (Frobenius) $\varphi: S \to S$ which extends
the Frobenius on $\gs$. Define a continuous $K_0$-linear derivation
$N: S \to S$ such that $N(u)= -u$.  We denote $S[1/p]$ by $S_{K_0}$.

Recall $R= \varprojlim \O_{\bar K }/p$ and the  unique surjective
map $\theta: W(R) \to \widehat{\O_{\bar K }}$  which lifts the
projection $R \to \O_{\bar K }/p$ onto the first factor in the
inverse limit. We denote by $A_{\t {cris}}$ the $p$-adic completion
of the divided power envelope of $W(R)$ with respect to
$\t{Ker}(\theta)$. Recall that $[\p] \in W(R)$ is the Teichm\"uller
representative of $\p = (\pi_n)_{n\geq 0} \in R$ and We embed the
$W(k)$-algebra $W(k)[u]$ into $W(R)$ via $u \mapsto [\p]$. Since
$\theta(\p)= \pi$, this embedding extends to an embedding $\gs \inj
S \inj A_{\t{cris}}$, and $\theta |_S$ is the $K_0$-linear map $s: S
\to \O_K $ defined by sending $u$ to $\pi$. The embedding is
compatible with Frobenius endomorphisms. As usual, we write
$B_\t{cris}^+= A_\t{cris}[1/p]$.


 For any field extension $F/\Q_p$, set $F_{p^\infty}=\bigcup \limits_{n=1}^\infty F
(\zeta_ {p^n})$ with $\zeta _{p^n} $ a primitive $p^n$-th root of
unity. Note that $K_{\infty, p^\infty}= \bigcup \limits
_{n=1}^\infty K({\pi_n}, \zeta_{p^n})$ is Galois over $K$. Let $G_0
:= \gal(K_{\infty, p^\infty}, K_{p^\infty})$, $H_K:= \gal
(K_{\infty, p^\infty}, K_\infty)$ and $\hat G: =\gal (K_{\infty,
p^\infty}/K) $.
 By  Lemma 5.1.2 in
\cite{liu3}, we have $K_{p^\infty} \cap K_\infty = K$,  $\hat G=
G_0 \rtimes H_K$ and $G_0 \simeq \Z_p(1)$.

For any $g \in G$, let $\e(g)= {g([\p])}/{[\p]}$. Then $\e(g)$ is a
cocycle from $G$ to the group of  units of $\acris$. In particular,
fixing a topological generator $\tau$ of $G_0$, the fact that $\hat G=
G_0 \rtimes H_K$  implies that
$\e(\tau)= [(\epsilon_{i})_{i \geq 0}] \in W(R)$ with $\epsilon_{i}$
a \emph{primitive} $p^i$-th root of unity. Therefore, $t := -\log(\e(\tau))
\in \acris$ is well defined and for any $g \in G$, $g(t) = \chi(g)t$
where $\chi$ is the cyclotomic character.

For any integer $n\geq 0$, let $t^{\{n\}}= t ^{r(n)} \gamma_{\tilde
q(n)}(t^{p-1}/p)$ where $n = (p-1)\tilde q(n) + r(n)$ with $ 0 \leq
r(n)< p-1$ and $\gamma_i(x) = \frac{x^i}{i!}$ is the standard
divided power. Define a subring $\R_{K_0}$ of $ B^+_\cris$ as in
\S6, \cite{liu2}:
  $$\R_{K_0}=\{x = \sum_{i=0 }^\infty f_i t^{\{i\}}, f_i \in S_{K_0}
    \t{ and } f_i \to 0\t{ as }i \to +\infty \}.$$
Finally we put $\hat \R := \R_{K_0} \cap W(R)$.

It is easy to see that $\R_{K_0}$ is an $S$-algebra and $\varphi$-stable as a subring of $B^+_\cris$. We claim that $\R_{K_0}$ is also $G$-stable.
In fact, it suffices to show that for any $g \in G$, $x \in S$, we have $g(x) \in \R_{K_0}$.
First note that for any $g \in G_\infty$, $g(x)= x$. Recall that $\hat G=G_0 \rtimes H_K$ and $\tau$
the fixed topological generator in $G_0$. It suffices to check that $\tau(x) \in \R_{K_0}$. But we have
$$\tau (x) := \sum \limits_{i=0}^\infty N^i(x) \gamma_i(-\log(\e(\tau))) = \sum_{i=0}^\infty N^i(x) \gamma_i (t).$$
Therefore $\tau(x) \in \R_{K_0}$ and $\R_{K_0}$ is $G$-stable. In fact, the $G$ action on $\R_{K_0} $
factors through $\hat G$. Hence we obtain some elementary facts on $\hR$.
\begin{lemma} \begin{enumerate}\item  $\hat \R$ is a $\varphi$-stable $\gs$-algebra as a subring in $W(R)$.
\item $\hat \R$ is $G$-stable. The $G$-action on $\hR $ factors though $\hat G$.
\end{enumerate}
\end{lemma}
\begin{remark} By Lemma 7.1.2 in \cite{liu2}, we may regard $R_{K_0}$ as a subring of $K_0[\![x, y]\!]$ via $u \mapsto x $ and $t \mapsto y$.
However, the structure of $\hR$ is much more complicated and so far we do not know how to describe it explicitly. See Example \ref{example}.
\end{remark}

Let $(\M , \varphi_\M )$ be a Kisin module of height $r$ and $\hat
\M := \hR \otimes_{\varphi,\gs} \M$. Then we can naturally extend
$\varphi$ from $\M$ to $\hat \M$ by
$$\varphi_{\hat \M} (a \otimes m) = \varphi_{\hR}(a) \otimes \varphi_\M (m), \ \ \ \forall a \in \hR,\ \forall m \in \M.  $$
\begin{definition}\label{definition} A $(\varphi, \hat G)$-module (of height $r$) is a triple $(\M , \varphi, \hat G)$ where
\begin{enumerate} \item $(\M, \varphi_\M)$ is a Kisin module (of height $r$).
\item $\hat G$ is a $\hR$-semi-linear $\hat G$-action on $\hat \M:=\hR \otimes_{\varphi, \gs} \M$.
\item $\hat G$ commutes with $\varphi_{\hM}$ on $\hM$, i.e., for any $g \in \hat G$, $g \varphi_{\hM} = \varphi_{\hM} g$.
\item Regard $\M$ as an $\varphi(\gs)$-submodule in $ \hM $,  then $\M \subset \hM ^{H_K}$.
\end{enumerate}
\end{definition}

 A morphism between two $(\varphi, \hat G)$-modules is a morphism of Kisin modules and commutes with $\hat G$-action  on $\hM$'s.
  We denote by $\t{Mod}^{r , \hat G}_{/\gs}$ the category of $(\varphi, \hat G)$-modules of height $r$.

\subsection{The main theorem} Let $\hM=(\M, \varphi, \hat G)$ be a $(\varphi, \hat G)$-module. We can associate a $\Z_p[G]$-module:
\begin{equation}\label{hatT}
\hat T (\hM) := \t{Hom}_{\hR, \varphi}(\hR \otimes_{\varphi, \gs}
\M, W(R)),
\end{equation}
where $G$-acts on $\hat T(\hM)$ via $g (f)(x) = g (f(g^{-1}(x)))$
for any $g \in G$ and $f \in \hat T(\hM)$. Now we can state our main
theorem:
\begin{theorem}\label{main}  \begin{enumerate}\item  Let $\hM:=(\M, \varphi, \hat G) $ be a $(\varphi, \hat G)$-module.
There is a natural isomorphism of $\Z_p[G_\infty]$-modules
\begin{equation}\label{isom}
\theta: T_\gs(\M)= \t{Hom}_{\gs ,\varphi} (\M,\gs^\ur)
\overset{\sim}{\longrightarrow} \hat T (\hM)= \t{Hom}_{\hR,
\varphi}(\hR \otimes_{\varphi, \gs}  \M, W(R)).
\end{equation}
\item $\hat T$ induces an anti-equivalence between the category of $(\varphi, \hat G)$-modules of height $r$ and the category
of $G$-stable $\Z_p$-lattices in semi-stable representations with Hodge-Tate weights in $\{0, \dots, r\}$.
\end{enumerate}
\end{theorem}

\section{The Proof of the Main Theorem}
\subsection{The connection to Kisin's theory} We first prove Theorem \ref{main} (1) and full faithfulness of $\hat T$ in this subsection.

Let $(\M, \varphi, \hat G)$ be a $(\varphi, \hat G)$-module and $\hM
:= \hR \otimes_{\varphi, \gs} \M$. As in Definition
\ref{definition}, we regard $\M$ as a $\varphi(\gs)$-submodule of
$\hM$. Then for any $f \in T_\gs(\M)$, define $\theta(f)\in
\t{Hom}_{\hR} ( \hM, W(R))$ by
$$ \theta(f)(a \otimes x): = a \varphi(f(x)), \ \  \forall a \in  \hR , \ \forall x \in \M. $$
It is routine to check that $\theta (f)$ is well-defined and preserves Frobenius. Therefore, $\theta : T_\gs(\M) \to \hat T(\hM)$ is a well-defined. Now we reduce the proof of Theorem \ref{main} (1) to the following


\begin{lemma} \label{compatibe} $\theta: T_\gs(\M) \to \hat T(\hat \M)$ is an isomorphism of $\Z_p[G_\infty]$-modules.
\end{lemma}
\begin{proof} Since $\varphi: \gs^\ur \to W(R) $ is injective, $\theta$ is obviously an injection. To see that $\theta$ is surjective,
for any $h \in \hat T(\hM)$, consider $f:= h|_{\M}$. Since $f$ is a
$\varphi(\gs)$-linear morphism from $\M$ to $W(R)=\varphi(W(R))$.
There exists an $\mathfrak f \in  \t{Hom}_{\gs}(\M, W(R))$ such that
$\varphi(\mathfrak f) = f$. Obviously, $\theta (\mathfrak f) = h$
and $\mathfrak f$ preserves Frobenius. Now we have $\f \in
\t{Hom}_{\gs, \varphi}(\M, W(R))$. It suffices to show that $\f \in
T_\gs(\M)= \t{Hom}_{\gs, \varphi}(\M, \gs ^\ur)$. Note that $f(\M)
\subset W(R)$ is an $\gs$-finite type $\varphi$-stable submodule  and
of  $E(u)$-height $r$. By \cite{fo4}, Proposition B 1.8.3, we have
$f(\M) \subset \gs ^\ur$. This complete the proof of the bijection
of $\theta$. Now it suffices to check that $\theta$ is compatible
with $G_\infty$-actions on the both sides. For any $g \in G_\infty$,
$a \in \hR$, $x \in \M$ and $f \in T_\gs(\M)$,  $g(\theta(f))(a
\otimes x)= g (\theta(f)(g^{-1}(a\otimes x)))$. Note that $G_\infty$
acts on $\M$ trivially,  we have
$$g (\theta(f)(g^{-1}(a\otimes x)))= g (\theta(f)(g^{-1} (a) \otimes x))= a \otimes g(\varphi(f(x))) =
 \theta (g(f))(a \otimes x).  $$
That is, $g(\theta (f))= \theta (g(f))$.
\end{proof}

Now we need some preparations to show that $\hat T(\hat \M)$ is
semi-stable. Let $T$ be a finite free $\Z_p$-representation of $G$
or $G_\infty$, we denote by $T^\v$ the $\Z_p$-dual of $T$. It will
be useful to recall the following technical results from
\cite{liu2}, \S3.2: let $\M$ be a Kisin module of height $r$, using
the definition of $T_\gs(\M)$, we can show (c.f.  \cite{liu2},
Proposition 3.2.1) there exists an $\gs^\ur$-linear,
$G_\infty$-compatible morphism\footnote{Here we use a slightly
different notations from those in \cite{liu2}.}
$$\iota_\gs: \gs^\ur\otimes_\gs \M \to T^\v _\gs(\M) \otimes _{\Z_p}\gs^ \ur .$$
Select a $\gt \in \gs^\ur$ such that $\varphi(\gt)= c_0^{-1}E(u)
\gt$ where $c_0$ is the constant term of $E(u)$. Such $\gt$ is
unique up to  units of $\Z_p$, see Example 2.3.5 in \cite{liu2} for
details.
\begin{lemma} \label{denominator} $\iota_\gs$ is an injection. If we regard
 $\gs^\ur\otimes_\gs \M$ as a submodule of $ T^\v _\gs (\M) \otimes _{\Z_p}\gs^ \ur$ via $\iota_\gs $. Then
$\gt^r (T^\v _\gs(\M) \otimes _{\Z_p}\gs^ \ur) \subset \gs^\ur\otimes_\gs \M.$
 \end{lemma}
\begin{proof} See Theorem 3.2.2 in \cite{liu2}.
\end{proof}
 Using the same idea as above,  we have a similar result for $\hM$
 \begin{prop}\label{comphat}
 $\hat T (\hM)$ induces a natural $W(R)$-linear, $G$-compatible morphism
\begin{equation}\label{iotahat }
\hat \iota: \  W(R)\otimes_{\hR} \hM  \longrightarrow \hat T^\v (\hat \M)
\otimes_{\Z_p} W(R),
\end{equation}
 where $\hM = \hR \otimes_{\varphi, \gs} \M $. Moreover,  $\hat \iota = \iota_\gs  \otimes _{\gs^\ur, \varphi  } W(R)$ and $\hat \iota$ is an injection. If  we regard
 $W(R)\otimes_{\hR} \hM$ as a submodule of $ \hat T^\v (\hM) \otimes _{\Z_p}W(R)$ via $\hat \iota $. Then
$(\varphi(\gt))^r (\hat T^\v(\hM) \otimes _{\Z_p}W(R)) \subset W(R)\otimes_{\hR} \hM.$
 \end{prop}
 \begin{proof}
 We use the same idea for the construction of  $\iota_\gs$  in Proposition 3.2.1 in \cite{liu2}. One first prove that
 $$\hat T (\hM) \simeq \t{Hom}_{W(R), \varphi} (W(R) \otimes_{\hR} \hM, W(R)) $$
 is an isomorphism of $G$-modules, where the $G$-action on the right side is given by $g(f)(\cdot)= g(f(g^{-1}(\cdot)))$,  for any
  $g \in G$ and $f \in \t{Hom}_{W(R), \varphi} (W(R) \otimes_{\hR} \hM,
  W(R))$. Then we have a map $$\hat \iota:\  W(R) \otimes_{\hR} \hM \to \t{Hom}_{\Z_p}(\hat T(\hM), W(R))=\hat T^\v(\hM)\otimes_{\Z_p}
  W(R)$$
  induced by  $x \mapsto (f\mapsto f(x), \  \forall f \in \hat
  T(\hM))$ for any $x \in \hM$. It is easy to check that $\hat
  \iota$ is compatible with $G$-actions on the both sides. By Lemma
  \ref{compatibe} and comparing the constructions of $\iota_\gs$ and $\hat \iota$,
   we see that $\hat \iota= \iota_\gs \otimes_{\gs^\ur, \varphi  }
   W(R)$. The remaining statements are then easy consequences of Lemma
   \ref{denominator}.
 \end{proof}
\begin{remark}Let $V$ be a representation of  $E(u)$-height $r$,  $T$  a $G$-stable $\Z_p$-lattice in
$V$, and $\M$ the Kisin module associated to $T|_{G_\infty}$. We can
always consider the injection $$\tilde \iota: = \iota_\gs \otimes_{
\gs^\ur, \varphi }W(R): \ W(R)\otimes_{\varphi, \gs} \M \inj
T^\v_\gs (\M)\otimes_{\Z_p}W(R).$$ There is a natural $G$-action on
the right side because $T$ is $G$-stable. In general, it is not
clear whether the left side is $G$-stable\footnote{Though it is $G_\infty$-stable.}, or equivalently,  whether the $ G$-orbit of
$\M$, $ G(\M)\subset W(R)\otimes_{\varphi, \gs}\M$. As we will
see soon, in the case of $(\varphi, \hat G)$-modules, we have $
G(\M)\subset \hR\otimes_{\varphi, \gs}\M \subset
W(R)\otimes_{\varphi, \gs}\M$. This is actually a key point to prove
that $\hat T(\hM) \otimes_{\Z_p}\Q_p$ is semi-stable.
\end{remark}
 Now we are ready to prove that $\hat T(\hM)
\otimes_{\Z_p}\Q_p$ is semi-stable.
 Tensoring $B^+_\cris$ on both sides of \eqref{iotahat },   noting that
 $$B^+_\cris \otimes_{W(R)}W(R) \otimes_{\hR}\hM= B^+_\cris\otimes_{\hR}\hM=B^+_{\cris}
 \otimes_{\R_{K_0}} \R_{K_0}\otimes_{\hR} \hM  $$ and  $ \R_{K_0}\otimes_{ \hR} \hM   \simeq \R_{K_0} \otimes_{\varphi, \gs} \M$, we have
\begin{equation}\label{comp1}\hat\iota \otimes_{W(R)} B^+_\cris:\  B^+_\cris \otimes_{\R_{K_0}}(\R_{K_0} \otimes_{\varphi,\gs}  \M
) \to \hat T^\v (\hM)\otimes_{\Z_p} B^+_{\cris}.
\end{equation}
 By a similar argument for $\iota_\gs$, we also have
\begin{equation}\label{comp2}\iota_\gs \otimes_{\gs^\ur, \varphi} B^+_\cris:\
B^+_\cris \otimes_{\R_{K_0}}(\R_{K_0} \otimes_{\varphi,\gs}  \M )
\to T^\v_\gs (\M)\otimes_{\Z_p} B^+_{\cris}.
\end{equation}
  Since
$\hat \iota = \iota_\gs  \otimes _{\gs^\ur, \varphi }W(R)$ by
Proposition \ref{comphat}, we have the following commutative diagram
to identify \eqref{comp1} with \eqref{comp2}:
$$\xymatrix{B^+_\cris \otimes_{\R_{K_0}}(\R_{K_0} \otimes_{\varphi,\gs}  \M )
 \ar[r]^-{\eqref{comp1}}\ar[d]^{\wr} & \hat T^\v (\hM)\otimes_{\Z_p} B^+_{\cris}\ar[d]^\wr\\
B^+_\cris \otimes_{\R_{K_0}}(\R_{K_0} \otimes_{\varphi,\gs}  \M )
\ar[r]^-{\eqref{comp2}}&T^\v_\gs (\M)\otimes_{\Z_p} B^+_{\cris} }
$$
Thus $\R_{K_0}\otimes_{\varphi, \gs} \M$ in \eqref{comp2} is
$G$-stable and has the same $\hat G$-action as that on
$\R_{K_0}\otimes_{\varphi, \gs} \M$ in \eqref{comp1}. Now the proof
of semi-stability of $\hat T (\hM)$ will be totally same as
\cite{liu2}, \S 7. For convenience of readers, we sketch the proof
here.

 \cite{liu2}, \S7 is also aiming to
prove that certain representation $V$ of  $E(u)$-height $r$ is
semi-stable with Hodge-Tate weights in $\{0, \dots, r\}$. Except
that we require that $V$ is of finite $E(u)$-height such that we can
establish \eqref{comp2}, the only other inputs that \S 7 need are
three conditions in the beginning of \S7.1 on the $\hat G$-action on
$\D \otimes_S \R_{K_0}\simeq \R_{K_0} \otimes_{\varphi, \gs}\M$,
where $\D := S_{K_0} \otimes_{\varphi, \gs}\M$. And these three
conditions are just conditions required in Definition
\ref{definition}. Thus the same proof follows. More precisely,
regarding $S_{K_0}$ as a subring of $K_0[\![u]\!]$, let $I= u
K_0[\![u]\!] \cap S_{K_0}$ and $D:= \D/I \D$. Then $D$ is a  finite
free $K_0$-module with a semi-linear Frobenius action. One can prove
there is a unique $\varphi$-equivariant section $D \inj \D$. So we
can regard $D$ as a $K_0$-submodule in $\D$. Since $D\inj \D$ is
$\varphi$-equivariant, the structure of $\R_{K_0}$ forces that $\hat
G(D) \subset K_0[t] \otimes_{K_0}D$. Now the fact that $\hat G$ acts
on $K_0[t] \otimes_{K_0}D$ and $H_K$ acts on $D$ trivially implies
that there exists a linear map $N: D \to D$ such that $\tau(x) =
\sum \limits_{n=0}^\infty \gamma_i(t) \otimes N^i(x) $ for any $x
\in D$. Now consider the $K_0$-vector space
$$\bar D:= \{\sum_{i=0}^ \infty \gamma_i (\mathfrak u) \otimes N^i(x) \in B^+_\st \otimes_S \D  | x \in D  \}$$
where $\mathfrak u = \log(u) \in B^+_\st$.
 We can
show that $\bar D \subset (B^+_\st \otimes_S \D)^G \subset (\hat
T^\v(\hM) \otimes_{\Z_p} B^+_\st)^G$. But $\t{dim}_{K_0}\bar D =
\t{dim}_{K_0}D = \t{Rank}_{\Z_p} \hat T(\hM)$. Therefore, $\hat
T(\hM)\otimes_{\Z_p}\Q_p$ is semi-stable and the functor $\hat T$ is
well-defined.

Now let us prove the full faithfulness of $\hat T$. Suppose that
$f: T'\to T $ is a morphism of $G$-stable $\Z_p$-lattices inside
semi-stable representations, and there exist $(\varphi, \hat
G)$-modules $\hM'$ and $\hM$ such that $\hat T (\hM')\simeq T'$ and
$\hat T(\hM)\simeq T$. Note that $T_\gs$ is fully faithful (Theorem
\ref{kisin}), there exists a morphism of Kisin modules $\mathfrak f
: \M \to \M'$ such that $T_\gs (\mathfrak f)= f|_{G_\infty}$, where
$\hM =  \hR \otimes_{\varphi, \gs} \M$ and $\hM' = \hR
\otimes_{\varphi, \gs}\M'$. By Lemma \ref{compatibe}, it suffices to
show that $\hat{\mathfrak f}:= \hR \otimes_{\varphi, \gs}\mathfrak f
$ is  $G$-equivariant. To see this, consider the following
commutative diagram induced by  $\hat \iota $ defined in
\eqref{iotahat }:
$$\xymatrix{ W(R) \otimes_{\hR}\hM\ar[d]^{W(R)\otimes_{ \hR}\hat \f} \ar[r]^-{\hat
\iota}&T^\v(\hM)\otimes_{\Z_p}W(R)\ar[d]^{T_\gs(\f)\otimes_{\Z_p}W(R)} \ar@{=}[r]&T^\v \otimes_{\Z_p}W(R)\ar[d]^{f\otimes_{\Z_p}W(R) }\\
W(R) \otimes_{\hR}\hM' \ar[r]^-{\hat
\iota'}&T^\v(\hM')\otimes_{\Z_p}W(R)\ar@{=}[r]&(T')^\v
\otimes_{\Z_p}W(R)}
$$
Note that $\hat \iota$ is injective by Proposition \ref{comphat}.
Since $f$ is $G$-equivariant, $\hat{\mathfrak f}$ is
$G$-equivariant.

\subsection{The essential surjectiveness of $\hat T$ }
Now assume $T$ is a $G$-stable $\Z_p$-lattice in a semis-stable
representation $V$ with Hodge-Tate weights in $\{0, \dots, r\}$. By
Theorem \ref{kisin}. There exists a Kisin module $\M$ such that
$T_\gs(\M) \simeq T|_{G_\infty}$. Theorem 5.4.2 in \cite{liu2}
showed that
$$\iota_\gs \otimes_{\gs^\ur,  \varphi} B^+_\cris:\
 {B^+_\cris}\otimes_{\varphi, \gs}\M \longrightarrow T^\v_\gs(\M) \otimes_{\Z_p}  B^+_\cris= T^\v \otimes_{\Z_p} B^+_\cris$$
 is compatible with $G$-action. More precisely, let $\D := S_{K_0}
\otimes_{\varphi, \gs}\M$ be the Breuil module\footnote{A Breuil
module is a finite free $S_{K_0}$-module with structures of
Frobenius, filtration and monodromy. By \cite{b2}, the category of
admissible Breuil modules is equivalent to the category of
semi-stable representations. Also see \S3.2 in \cite{liu3} for the
relation between Kisin modules and Breuil modules.} associated to
$V$ and $N$ be the monodromy operator on $\D$. Then $G$ acts on
$B^+_\cris \otimes_{S} \D= {B^+_\cris}\otimes_{\varphi, \gs}\M$ via
(c.f. (5.2.1) in \cite{liu2})
\begin{equation}\label{action}
g(a \otimes x)= \sum_{i=0}^\infty g(a)\gamma _i (-\log(\e(g)))
\otimes N^i (x) .
\end{equation}
We identify $\M$ as a $\varphi(\gs)$-submodule of
${B^+_\cris}\otimes_{\varphi, \gs}\M$ by
$$\M \simeq \gs \otimes _\gs \M \inj \gs \otimes _{\varphi, \gs} \M \inj {B^+_\cris}\otimes_{\varphi, \gs }\M.$$
Now we are interested in the orbit $G(\M)$ of $\M$ under $G$.
\begin{prop}\label{stable}
$G(\M)\subset \hR \otimes_{\varphi, \gs}\M$.
\end{prop}
\begin{proof}We have seen that $G(\M) \subset \R_{K_0} \otimes_{\varphi, \gs}\M$ by formula \eqref{action}.
Then it suffices to show that $\tau(\M) \subset
W(R)\otimes_{\varphi, \gs}\M$, where $\tau$ is the fixed generator
of $G_0$. Now consider the following commutative diagram
\begin{equation}\label{phi-block}
\begin{split} \xymatrix{\M \ar@{=}[d]\ar@{^{(}->}[r]^-{\varphi\otimes 1} & W(R) \otimes_{\varphi, \gs} \M
\ar@{_{(}->}[d] \ar[r] &
T^\v \otimes_{\Z_p} W(R)\ar@{^{(}->}[d] \\
\M \ar@{^{(}->}^-{\varphi\otimes 1}[r] & B^+_\cris
\otimes_{\varphi,\gs} \M \ar[r] & T^\v \otimes _{\Z_p} {B^+_\cris}}
\end{split}
\end{equation}
where the first row is obtained by $\iota_\gs \otimes _{\gs^\ur,
\varphi} W(R)$. Obviously, the right column is compatible with
$G$-action. By Proposition \ref{comphat},  we have
$$(\varphi(\gt))^r (\tau(\M)) \subset W(R)
\otimes_{\varphi,\gs} \M.$$ Now select a basis $e_1, \dots , e_d $
of $\M$ and write $\tau(e_1, \dots, e_d) = (e_1, \dots, e_d)A$ with
$A$ a $d\times d$-matrix. Let $a$ be a coefficient of $A$. It
suffices to show that $a \in W(R)$. Now we know that $a \in
\R_{K_0}$ and $(\varphi(\gt))^r a \in W(R)$. Then we may reduce the
proof to Lemma \ref{key} below.

\end{proof}

\begin{lemma}\label{key} Let $a \in B^+_\t{cris}$. If $(\varphi(\gt))^r a \in W(R)$ then $a \in W(R)$.
\end{lemma}
\begin{proof}After multiplying some $p$-power, we may assume that $a \in \acris$. As in  \cite{fo3}, \S5.1,  define
$$I^{[r]}W(R)=\{ a\in W(R)| \varphi^n (a) \in \t{Fil}^r W(R), \t{ for any } n \geq 0\}.$$
Write $x= (\varphi(\gt))^r a $. We claim that $x \in I^{[r]}W(R)$.
By Example 5.3.3 in \cite{liu2} or Example \ref{example} below,
there exists a unit $\alpha \in \acris$ such that  $t= \alpha
\varphi(\gt)$. So $\varphi(\gt)\in \t{Fil}^1 W(R) $, then $
x=(\varphi(\gt))^r a \in \t{Fil}^r \acris \cap W(R)= \t{Fil}^r W(R)$. On the other hand,
 $\varphi(\gt)= c_0^{-1}E(u)\gt$, thus $\varphi^n(\gt)=
(\prod\limits_{i=0}^{n-1} \varphi^i(c_0^{-1}E(u)))\gt \in \t{Fil}^1
W(R)$. Therefore $\varphi^n(x), \varphi^{n}((\varphi(\gt))^r)\in\t{Fil}^r W(R) $  and then $x, (\varphi(\gt) )^r \in
I^{[r]}W(R)$. By proposition 5.1.3 in \cite{fo3}, $I^{[r]}W(R)$ is a
principal ideal and  $b \in I^{[r]}W(R)$ is a generator if and only
if $v_R(\tilde b) = \frac{rp}{p-1}$, where $\tilde b= b \mod p$. We
claim that $(\varphi(\gt))^r$ is a generator of $I^{[r]}W(R)$ by
computing $v_R(\widetilde{\varphi(\gt)})= \frac{p}{p-1}$. Since
$\varphi (\gt)= c_0^{-1}E(u)\gt$, we may choose $\gt$ such that
$\tilde \gt =\gt \mod p = \underline\pi^{\frac{e}{p-1}}$.  Thus
$v_R(\tilde \gt)= \frac{1}{p-1}$. Now $(\varphi(\gt))^r a \in
I^{[r]}W(R) $ and $(\varphi(\gt))^r$ is a generator of $I^{[r]}W(R)$. So
$a \in W(R)$.
\end{proof}
Now Proposition \ref{stable} implies that $\hM:= \hR
\otimes_{\varphi, \gs} \M $ is stable under $G$-action in
${B^+_\cris}\otimes_{\varphi, \gs} \M \inj T^\v \otimes_{\Z_p}
B^+_\cris$. And obviously the $G$-action on $\hM$ factors through
$\hat G$. It is easy to check that $\hM$ is a $(\varphi, \hat
G)$-module. It remains to check that $\hat T(\hM) \simeq T$. First,
by Lemma  \ref{compatibe}, $\hat T(\hM)|_{G_\infty} \simeq
T|_{G_\infty}$. Recall that $\hat \iota$ defined in \eqref{iotahat }
is compatible with the $G$-action on the both sides, and $\hat
\iota= \iota_{\gs} \otimes_{\gs, \varphi} W(R)$. Comparing $\hat
\iota$ with the top row of \eqref{phi-block}, we have the following
commutative diagram:
$$\xymatrix{  W(R)\otimes_{\hR} \hM
\ar[d]^{\wr} \ar[r]^-{\hat \iota} &
\hat T^\v(\hM) \otimes_{\Z_p} W(R)\ar[d]_{\wr } \\
W(R) \otimes_{\varphi, \gs} \M  \ar[r] &  T^\v \otimes_{\Z_p} W(R)
}$$ By the construction of $\hM$, we see that the left column is
compatible with the $G$-actions.  By Proposition \ref{comphat},
$\varphi(\gt^r) (T^\v_\gs(\M) \otimes _{\Z_p} W(R))\subset
W(R)\otimes_{\varphi,\gs} \M$. So the right column is also
compatible with the $G$-actions. Therefore, $\hat T(\hM) \simeq T$
as $G$-modules. This completes the proof the main theorem.

Unlike $\R_{K_0}$, so far we do not have an explicit description of
$\hR$. Put  $$w:= \frac{\tau(u)}{u}-1= \frac{\tau([\p])}{[\p]}-1=
\exp(-t)-1.$$ We see that $w \in \hR$ and $W(k)[\![u, w]\!]\subset \hR$
is stable under Frobenius and $\hat G$-action. Unfortunately, this
inclusion is strict. The following example show that the structure
of $\hR$ may be very complicated.
\begin{example} \label{example}It is well known that $t$ is the period of the
cyclotomic character $\chi$. On the other hand, $\gt$ is the period
of  the Kisin module for $\chi$, where is $\gs$-free rank-$1$ module
$\gs ^* := \gs\cdot f$ and $\varphi(f)= c_0^{-1}E(u)f$ with $f$ a
basis. Example 5.3.3 in \cite{liu2} showed that we may choose $t$
such that $t = c\varphi(\gt)$, where $c =\prod \limits_{i=0}^\infty
\varphi^n(\frac{\varphi(c^{-1}_0E(u)}{p})$. Then $\tau (c)\tau
(\varphi (\gt))= \tau (t)= t = {c}\varphi(\gt) $. Therefore $\tau
(\varphi(\gt))= \frac{c}{\tau (c)} \varphi(\gt)= \prod
\limits_{n=1}^\infty \varphi^n(\frac{E(u)}{\tau(E(u))})
\varphi(\gt)$. Let $\hat c = \frac{c}{\tau(c)}= \prod
\limits_{n=1}^\infty \varphi^n(\frac{E(u)}{\tau(E(u))})$. Since $c$
is a unit in $\R_{K_0}$, $\hat c \in \R_{K_0}$. On the other hand,
$E(u)$ is a generator of $\t{Fil}^1 W(R)$, so $\frac{
E(u)}{\tau(E(u))}$ is a unit in $W(R)$. Thus $\hat c \in \hR$. Let
$\hat\gs^* := (\hR \otimes_{\varphi,\gs} \gs) \cdot f= \hR \cdot f$
be the $(\varphi, \hat G)$-module corresponding to $\chi$. Then
$\hat G$-action on $\hat \gs ^*$ is given by $\tau (f)= \hat c f$.
\end{example}

\bibliographystyle{amsalpha}

\bibliography{biblio1}

\providecommand{\bysame}{\leavevmode\hbox to3em{\hrulefill}\thinspace}
\providecommand{\MR}{\relax\ifhmode\unskip\space\fi MR }
\providecommand{\MRhref}[2]{%
  \href{http://www.ams.org/mathscinet-getitem?mr=#1}{#2}
}
\providecommand{\href}[2]{#2}
\begin{thebibliography}{Fon94b}

\bibitem[BB07]{Ber3}
Laurent Berger and Christophe Breuil, \emph{{S}ur quelques repr\'esentations
  potentiellement cristallines de $\textnormal{GL}_2( \mathbb{Q}_p)$},
  Preprint, avaliable at {\tt http://www.umpa.ens-lyon.fr/\~\ lberger/} (2007).

\bibitem[Bre97]{b2}
Christophe Breuil, \emph{Repr\'esentations {$p$}-adiques semi-stables et
  transversalit\'e de {G}riffiths}, Math. Ann. \textbf{307} (1997), no.~2,
  191--224.

\bibitem[Bre02]{b6}
\bysame, \emph{Integral {$p$}-adic {H}odge theory}, Algebraic geometry 2000,
  Azumino (Hotaka), Adv. Stud. Pure Math., vol.~36, Math. Soc. Japan, Tokyo,
  2002, pp.~51--80.

\bibitem[FL82]{fo1}
Jean-Marc Fontaine and Guy Laffaille, \emph{Construction de repr\'esentations
  {$p$}-adiques}, Ann. Sci. \'Ecole Norm. Sup. (4) \textbf{15} (1982), no.~4,
  547--608 (1983).

\bibitem[Fon90]{fo4}
Jean-Marc Fontaine, \emph{Repr\'esentations {$p$}-adiques des corps locaux.
  {I}}, The Grothendieck Festschrift, Vol.\ II, Progr. Math., vol.~87,
  Birkh\"auser Boston, Boston, MA, 1990, pp.~249--309.

\bibitem[Fon94a]{fo3}
\bysame, \emph{Le corps des p\'eriodes {$p$}-adiques}, Ast\'erisque (1994),
  no.~223, 59--111.

\bibitem[Fon94b]{fo7}
\bysame, \emph{Repr\'esentations {$p$}-adiques semi-stables}, Ast\'erisque
  (1994), no.~223, 113--184, With an appendix by Pierre Colmez, P\'eriodes
  $p$-adiques (Bures-sur-Yvette, 1988).

\bibitem[Kis06]{kisin2}
Mark Kisin, \emph{Crystalline representations and {$F$}-crystals}, Algebraic
  geometry and number theory, Progr. Math., vol. 253, Birkh\"auser Boston,
  Boston, MA, 2006, pp.~459--496.

\bibitem[Liu07a]{liu3}
Tong Liu, \emph{Lattices in semi-stable representations: proof of a conjecture
  of {B}reuil}, Preprint, appear at Compositio Mathematica (2007).

\bibitem[Liu07b]{liu2}
\bysame, \emph{Torsion $p$-adic {G}alois representations}, Preprint, appear at
  Ann. Sci. \'Ecole Norm. Sup. (4) (2007).

\end{thebibliography}


\end{document}